\documentclass[]{spie}  %>>> use for US letter paper
\usepackage[]{amsmath, amssymb, amsfonts, amsbsy, latexsym}
\usepackage{enumerate}
\usepackage{graphicx}
\usepackage{color}
\newtheorem{thm}{Theorem}[section]

\newtheorem{cor}[thm]{Corollary}
\newtheorem{lem}[thm]{Lemma}
\newtheorem{prop}[thm]{Proposition}

\newtheorem{defn}[thm]{Definition}

\newtheorem{ex}[thm]{Example}

\newtheorem{prob}[thm]{Problem}

\DeclareMathOperator{\TR}{tr}

\newcommand{\HSIP}[1]{\left\langle#1\right\rangle_{\scriptscriptstyle\! HS}}
\newcommand{\NORM}[1]{\left\|#1\right\|}
\newcommand{\HSNORM}[1]{\left\|#1\right\|_{\scriptscriptstyle HS}}

\newcommand\HS[1]{{\mathbb R}^{#1}}
\newcommand\SPH[2]{{\mathcal S}\left(\mathbb{R}^{\scriptscriptstyle #2}\right)} 
\newcommand\Q[3]{\mathcal{\bf Q}^{(#1)}_{{#3} } } 
\newcommand\K[3]{{\bf K}^{(#1)}_{{#3} } } 
\newcommand\ETE[4]{E_{#1,#2}\otimes E_{#3,#4}} 
\newcommand\U[2]{\mathcal O_{\scriptscriptstyle \mathbb {R}^{#2}}}
\newcommand\SCITI{{\frac{1}{m^2}} {\bf I}_m \otimes {\bf I}_m}
\newcommand\CAP[2]{\mathcal C_{#1} \left(#2 \right) }

\newcommand\A[2]{{\mathbf a}_{{#2}}} 
\newcommand\B[2]{{\mathbf b}_{{#2}}}  
 
\newcommand\D[3]{{\mathbf d}^{(#1)}_{{#3}}} 
\newcommand\VAR[1]{{\mathbf V}^{(#1)}}

\newcommand{\UNFs}[1]{{\Omega}_{n,m} {(\mathbb {R})}}  %used to denote equal norm frames
\newcommand{\Gr}[1]{\mu_{n,m} (\mathbb R)} %used to denote Grassmannian constant
\newcommand{\SA}[2]{\mathcal B_{\scriptscriptstyle\mathrm{SA}}( \mathbb{R}^{#2})}

\begin{document}

%%%%%%%%%%%%%%%
\title{Low frame coherence via zero-mean tensor embeddings}
\author{Bernhard G. Bodmann and John I. Haas IV\thanks{~This work was partially supported by NSF grant DMS-1412524 and the AMS-Simons Travel grant.}}
%%%%%%%%%%%%%%%

\maketitle

\begin{abstract}
This paper is concerned with achieving optimal coherence for highly redundant real 
%or complex 
unit-norm frames. As the redundancy grows, the number of vectors in the frame becomes too large to admit equiangular arrangements. In this case, other geometric optimality criteria need to be identified. To this end, we use an iteration of the embedding technique by Conway, Hardin and Sloane. As a consequence of their work, a quadratic mapping  embeds equiangular lines into a simplex in a real Euclidean space.
Here, higher degree polynomial maps embed highly redundant unit-norm frames to simplices in high-dimensional Euclidean spaces. We focus on the lowest degree
case in which the embedding is quartic.
\end{abstract}


%%%%%%%%%%%%%%%%%%%%%%%%%%
%%%%%%%%%%%%%%%%%%%%%%%%%%
%%%%%%%%%%%%%%%%%%%%%%%%%%
%%%%%%%%%%%%%%%%%%%%%%%%%%
%%%%%%%%%%%%%%%%%%%%%%%%%%
%%%%%%%%%%%%%%%%%%%%%%%%%%

\section{Introduction}
The construction of equiangular lines has a long history in the mathematical literature \cite{Haantjes1948, Rankin1955, vanLintSeidel1966, LemmensSeidel1973, Zauner1999, MR1984549, MR2021601, XiaZhouGiannakis2005, MR2890902, MR2921716, MR3150919, Fickus:2015aa}. If the number of unit-norm vectors
spanning these lines cannot be enlarged any more without changing the set of angles/distances between them, then these vectors constitute an example
of an optimal packing. Such packings have applications ranging from coding, fiber-optic or wireless communications to phase retrieval and quantum information theory
\cite{ bod_cas_edi_bal_2008, MR2142983, MR836025}. An analytic formulation of equiangular lines as solutions of an optimization problem is the so-called Welch bound \cite{Welch1974}.
It can be obtained by combining a mapping of Conway, Hardin and Sloane \cite{ConwayHardinSloane1996} with a spherical cap packing bound by Rankin~ \cite{Rankin1955}.

In an earlier work, the case of redundancy beyond the equiangular regime was addressed by combining the embedding by Conway, Hardin and Sloane with the  orthoplex bound, which
is saturated by the example of maximal sets of mutually unbiased bases. With the help of relative difference sets, previously unknown examples of Grassmannian packings could be constructed \cite{MR3557826}; for examples, see the tables of Refs.~\citenum{Fickus:2015aa, 2016arXiv161003142C} for instances of Grassmannian frames with redundancies varying between that of maximal equiangular frames and maximal mutually unbiased bases.

Here, we iterate the embedding to obtain higher degree polynomial maps that are used to embed specific unit-norm frames to simplices. As a result, we identify 
several cases of optimal packings.

%%%%%%%%%%%%%%%%%%%%%%%%%%
%%%%%%%%%%%%%%%%%%%%%%%%%%
%%%%%%%%%%%%%%%%%%%%%%%%%%
\section{Preliminaries}
\subsection{Frame Theory}
Let $\{e_j\}_{j=1}^m$ denote the canonical orthonormal basis for the Hilbert space 
$\mathbb R^m$.  A sequence of vectors $\mathcal F = \{f_j\}_{j=1}^n \subset \mathbb R^m$ is a {\bf (finite) frame} for $\mathbb R^m$ if it spans the entire Hilbert space. From now on, we reserve the symbols $m$ and $n$ to refer to the dimension of the span of a frame and the cardinality of a frame, respectively.  The {\bf redundancy} of a given frame is the ratio $\frac n m$.

A frame $\mathcal F = \{f_j\}_{j=1}^n$  is {\bf $a$-tight} if 
$$\sum_{j=1}^n f_j  f_j^*= a\, \mathbf{I}_m, \text{ for some } a>0$$ 
where $\mathbf{I}_m$ denotes the $m \times m$ identity matrix.
The frame is {\bf unit-norm} if each frame vector has norm $\|f_j\|=1$.

Given a unit-norm frame $\mathcal F = \{ f_j \}_{j=1}^n$, its {\bf frame cosines} are the elements of the set 
$$
\Theta_\mathcal F : =\{ |\langle f_j, f_l \rangle | : j \neq l \},
$$
 and we say that $\mathcal F$ is {\bf $k$-angular} if $|\Theta_\mathcal F| =k$ for some $k \in \mathbb N$.   If
 $\Theta_{\mathcal F}$ has only one element and $\mathcal F$ is tight, then we speak of an {\bf equiangular tight frame}.

Let $\UNFs{R}$ denote the space of unit-norm frames for $\mathbb R^m$  consisting of $n$ vectors.
Given any set of unit vectors, $\mathcal F = \{ f_j \}_{j=1}^n \subset \mathbb F^m$,  its {\bf coherence} is defined by
$$\mu(\mathcal F) = \max\limits_{j \neq l} |\langle f_j, f_l \rangle |.$$
We define and denote the {\bf Grassmannian constant} as
$$
\Gr{F} = \min\limits_{\scriptscriptstyle \mathcal F \in \UNFs{F}} \mu(\mathcal F).
$$ 
Correspondingly, a frame $\mathcal F \in \UNFs{F}$ is a {\bf Grassmannian frame} if
$$
\mu(\mathcal F) = \Gr{F}. 
$$

%%%%%%%%%%%%%%%%%%%%%%%%%%
%%%%%%%%%%%%%%%%%%%%%%%%%%
%%%%%%%%%%%%%%%%%%%%%%%%%%
\section{Zero-mean tensor embeddings}
Our path toward identifying certain optimal line packings involves a two step process.  First, we apply a norm-preserving map to the frame vectors, thereby embedding the frame into a higher dimensional real sphere.  For the second step, we interpret the embedded vectors as the centers of spherical caps (which we discuss below) and exploit the cap packing results of Rankin~\cite{Rankin1955}.  If a frame embeds into  an optimal cap packing and certain additional conditions are satisfied, then the minimal coherence of the lifted frame is verified by the isometric nature of the the embedding.

We begin by defining the aforementioned family of norm-preserving maps.  We denote the unit sphere in $\mathbb R^m$ by $\SPH{F}{m}$ and we
let $\SA{F}{m}$ denote the real vector space of self-adjoint linear maps on $\HS{m}$.
From here on, $\omega$ is a random vector with values in $\SPH{F}{m}$ and 
$\mathbb E$ denotes the expectation with respect to the underlying uniform probability measure
on $\SPH{F}{m}$.

\begin{defn}\label{def_Qt}
The {\bf first zero-mean tensor embedding} is defined and denoted by
$$
\Q{1}{F}{m}: \SPH{F}{m} \rightarrow \SA{F}{m}: x \mapsto x \otimes x^* - {\scriptstyle \frac{1}{m}} {\bf I}_m \, ,
$$
and for $t \in \mathbb N$, the {\bf $(t+1)$-th zero-mean tensor embedding}, $\Q{t+1}{F}{m}$, is defined  recursively by
$$
   \Q{t+1}{F}{m}: \SPH{F}{m} \rightarrow \SA{F}{m}^{\otimes 2^{t-1}} : x \mapsto \left( \Q{t}{F}{m}(x)\right)^{\otimes 2} - \mathbb E \left[
    \left(\Q{t}{F}{m}(\omega) \right)^{\otimes 2} \right] \, .
$$
\end{defn}
For brevity, we also refer the $t$-th zero-mean tensor embedding simply as {\bf the $t$-th embedding}.  
The purpose of subtracting the expected value is that, just as $\mathbb E[\omega]=0$,  the mean of the embedding vanishes, 
$$
   \mathbb E\left[ \Q{t}{F}{m}(\omega) \right] = 0 .
$$  
In comparison with the action of of taking simple tensor powers of $x\otimes x^*$, the dimension of the range of the embedding is reduced by subtracting the expectation,
as we show in the next theorem.  To simplify notation, for each $t \geq 2$, we write
$$
\VAR{t}:= \mathbb E\left[\left( \Q{t-1}{F}{m}(\omega)\right)^{\otimes 2}\right].
$$

\begin{thm}\label{th_orth_cond}
 If $t \ge 2$ and $x \in \mathbb F^m$, then
$
  \TR[ \Q{t}{F}{m}(x) \VAR{t} ] = 0 \, .
$
\end{thm}
\begin{proof}
We note that for any unitary $U$ on $\HS{m}$,
$$
  U^{\otimes 2^{t-1}} \VAR{t} (U^*)^{\otimes 2^{t-1}} 
  = \mathbb E\left[ \Q{t-1}{F}{m}(U \omega) \otimes \Q{t-1}{F}{m}(U\omega) \right]
  = \VAR{t}
$$
because $U\omega$ and $\omega$ are identically distributed.
This implies that 
\begin{align*}
 \TR\left[ \Q{t-1}{F}{m}(x)\otimes \Q{t-1}{F}{m}(x) \VAR{t} \right] 
 &= \TR\left[   \Q{t-1}{F}{m}(x)\otimes \Q{t-1}{F}{m}(x)
              U^{\otimes 2^{t-1}} \VAR{t}  (U^*)^{\otimes 2^{t-1}} \right] \\
  &= \TR\left[  \Q{t-1}{F}{m}(U^*x)\otimes \Q{t-1}{F}{m}(U^*x) \VAR{t} \right]
\end{align*}
and by averaging with respect to the choice of $U^*$ among all unitaries,
$$
   \TR\left[  \Q{t-1}{F}{m}(x)\otimes   \Q{t-1}{F}{m}(x) \VAR{t} \right] = \TR\left[ \left(\VAR{t}\right)^2 \right] \, .
$$ 
Consequently,
$$
  \TR\left[ \Q{t}{F}{m}(x) \VAR{t}\right] = \TR\left[ \left( \Q{t-1}{F}{m}(x)\otimes   \Q{t-1}{F}{m}(x)-\VAR{t}\right) \VAR{t} \right] 
= 0 \, .
$$
\end{proof}

The space of symmetric tensors in $({\mathbb R^d})^{\otimes 2}$
is of dimension $d(d+1)/2$, and with an additional orthogonality condition
the range is reduced to a subspace of dimension $d(d+1)/2-1=(d+2)(d-1)/2$.
Iterating this dimensionality bound yields a maximal dimension of the subspace containing the 
range of $\Q{t}{F}{m}$.
Accordingly, we define
and denote the {\bf first embedding dimension} by
$$
\D{1}{F}{m} 
: = \frac{(m+2)(m-1)}{2}
%\left\{  
%	\begin{array}{cc}
%		\scriptstyle
%		\frac{(m+2)(m-1)}{2} & \mathbb F = \mathbb R \\
%		\scriptstyle
%		m^2-1 & \mathbb F = \mathbb C \\
%	\end{array} 
%\right.
 ,
$$
and, for $t\in \mathbb N$, we define the  {\bf $(t+1)$-th embedding dimension} recursively by
$$
  \D{t+1}{F}{m}
 : = \frac 1 2 (\D{t}{F}{m}+2)(\D{t}{F}{m}-1) \, .
$$ 
For example, the {\bf second embedding dimension} is
$$
    \D{2}{F}{m} =\frac{ \left(m^2+m+2\right)\left(m^2+m-4\right)}{8} 
%\left\{  
%	\begin{array}{cc}
%		\scriptstyle
%		(m^2+m+2)(m^2+m-4)/8 & \mathbb F = \mathbb R \\
%		\scriptstyle
%		(m^2+1)(m^2-2)/2 & \mathbb F = \mathbb C \\
%	\end{array} 
%\right.
 ,
$$
and the {\bf third embedding dimension} is
$$
    \D{3}{F}{m} = \frac{\left((m-1)m(m+1)(m+2)+8\right) \left((m-1)m(m+1)(m+2)-16\right)}{128}
%\left\{  
%	\begin{array}{cc}
%		\scriptstyle
%		((m-1)m(m+1)(m+2)+8)((m-1)m(m+1)(m+2)-16)/128 & \mathbb F = \mathbb R \\
%		\scriptstyle
%		(m^4-m^2+2)(m^4-m^2-4)/8 & \mathbb F = \mathbb C \\
%	\end{array} 
%\right. 
. 
$$

In particular, Theorem~\ref{th_orth_cond} yields the following corollary.

\begin{cor}
The range of the map $\Q{t}{F}{m}$ is contained in a subspace of 
$\SA{F}{m}^{\otimes 2^{t-1}}$,
whose dimension is at most equal to
$
 %\D{t-1}{F}{m}(\D{t-1}{F}{m}+1)/2 -1 =  (\D{t-1}{F}{m}+2)(\D{t-1}{F}{m}-1)/2 
% =  
\D{t}{F}{m}\, .
$
\end{cor}

\section{Rankin's bound and achieving optimal coherence}

In this section, we show how to exploit Rankin's classical bounds for spherical cap packings~\cite{Rankin1955} and deduce the optimality properties of certain frames.
%As the reader might anticipate at this point, there is a close connection between the so-called {\it spherical cap packing problem} and the search for Grassmannian frames.  
Given $d \in \mathbb N$, a unit vector $x\in\SPH{R}{d}$ and a real number $\theta \in (0, \pi]$, we define and denote the {\bf spherical cap of angular radius $\theta$ centered at $x$} as
$$
\CAP{x}{\theta} :=\left\{ y\in \SPH{F}{d} : \langle x, y\rangle > \cos \theta  \right\},
$$
which is alternatively referred to as a {\bf $\theta$-cap} when the center is arbitrary.
Rankin considered following optimization problem.
\begin{prob}\label{rankin}
Given a fixed dimension, $d$, and a fixed angle, $\theta$, what is the largest number, $n$, of $\theta$-caps, $\left\{ \CAP{x_j}{\theta} \right\}_{j=1}^n$, 
that one can configure on the surface of $\SPH{R}{d}$ such that $\CAP{x_j}{\theta} \cap \CAP{x_{j'}}{\theta} = \emptyset$  for each $j, j' \in \{1,2,...,n\}$ with $j \neq j'$.
\end{prob}
As a partial solution, Rankin reformulated Problem~\ref{rankin} in terms of its inverse optimization problem, providing sharp upper bounds on the caps' angular radii,  completely solving the problem  whenever $n\leq 2d$. We phrase his results in terms of the inner products between the caps' centers.

\begin{thm}\label{thm_rankin}[Rankin;~\cite{Rankin1955}]
Given any positive integer $d$
and any set of $n$ unit vectors $\{v_1, v_2, \dots, v_n\}$ 
in $\mathbb R^d$, then
$$
     \max_{\substack{ j, l \in \{1,2,...,n\} \\ j\neq l} }
                                   \langle v_j, v_l \rangle \ge - \frac{1}{n-1}, \, 
$$
and if $n>d+1$, then it improves to
$$
    \max_{\substack{ j, l \in \{1,2,...,n\} \\ j\neq l} }
                                   \langle v_j, v_l \rangle \ge 0. 
$$
\end{thm}

\subsection{The first embedding}

The first embedding has been used in conjunction with Rankin's cap-packing results~\cite{Rankin1955} to characterize and construct numerous families of Grassmannian 
frames~\cite{ConwayHardinSloane1996, Fickus:2015aa, MR3557826, Appleby2009}.
In more detail, the vectors are mapped to a self-adjoint rank-one Hermitian and projected onto the orthogonal complement of the identity matrix.  
The inner product between the images of two unit vectors is a polynomial of the original inner product. As a consequence, under certain assumptions, optimal packings are
equivalent to packings on a Euclidean sphere, and the Rankin bound can be applied.

\begin{thm}\label{thm_t1}[Conway, Hardin and Sloane;~\cite{ConwayHardinSloane1996}]
If $\mathcal F$ is a unit-norm frame of $n$ vectors in $\mathbb{F}^m$ and $\{\Q{1}{F}{m}(f): f \in \mathcal F\}$ forms a simplex in a  subspace
of the real space of self-adjoint $m \times m$ matrices over $\mathbb F$, then $\mathcal F$ is a Grassmannian frame. 
Moreover, if $n > \D{1}{F}{m} + 1$, and the Hilbert-Schmidt inner product of any pair from $\{\Q{1}{F}{m}(f): f \in \mathcal F\}$ 
is non-positive, then $\mathcal F$ is a Grassmannian frame. 
\end{thm}
\begin{proof}
It is straightforward to verify that
$$
   x \mapsto \tilde{T}^{(1)}(x) = \frac{m}{m-1} \Q{1}{F}{m}(x)
$$
maps the unit sphere in ${\mathbb F}^m$ to the unit sphere in $\SA{F}{m}$. 
More generally, the inner products of frame vectors $f_j$ and $f_l$ are related by
$$
 \left\langle  \tilde{T}^{(1)}(f_j),  \tilde{T}^{(1)}(f_l) \right\rangle_{HS} = \frac{m}{(m-1)} \left(|\langle f_j, f_l \rangle |^2  - \frac 1 m \right) \, .
  $$
  Now applying Rankin's bound shows that if the Hilbert-Schmidt inner product  assumes the constant value $\langle  \tilde{T}^{(1)}(f_j),  \tilde{T}^{(1)}(f_l) \rangle_{HS}=-\frac{1}{n-1}$
  when $j \ne l$, then the maximal magnitude occurring among inner products of pairs of vectors from $\mathcal F$
  is minimized.
  
  Moreover, if $n > \D{1}{F}{m} + 1$ and the maximum $\max_{j \ne l}  \left\langle  \tilde{T}^{(1)}(f_j),  \tilde{T}^{(1)}(f_l) \right\rangle_{HS} \le 0$,
  then by Rankin's bound equality holds and the frame is Grassmannian.
\end{proof}

Converting between the (squared) inner product of the frame vectors and the Hilbert-Schmidt inner product of the embedded vectors gives
the Welch and orthoplex bounds as consequence.

\begin{cor}
If $\mathcal F$ is a unit-norm frame of $n$ vectors in $\mathbb{F}^m$, then $\max_{j \ne l} |\langle f_j , f_l \rangle | \ge \sqrt{\frac{n-m}{(n-1)m}}$
and if equality holds, $\mathcal F$ is an equiangular tight frame. Moreover, if $n > \D{1}{F}{m}+1$, then
$\max_{j \ne l} |\langle f_j , f_l \rangle | \ge \frac{1}{\sqrt m}$ and if equality holds, $\mathcal F$ is a Grassmannian frame.
\end{cor}

\subsection{The second embedding}
  For the remainder of this work, 
we focus on the development of the analogous machinery corresponding to the second tensor embedding.  
In order to provide an explicit expression for second  embedding,  
we must compute the expectation,  $\mathbb E \left[
    \left(\Q{1}{F}{m}(\omega) \right)^{\otimes 2} \right]$, as given in Definition~\ref{def_Qt}.
To facilitate this, we define and denote the {\bf $t$-coherence tensor (for $\mathbb R^m$)} by 
$$
\K{t}{F}{m} 
:= 
\int\limits_{\U{F}{m}} \left( UPU^* \right)^{\otimes t} d\mu(U),
$$
where $\U{F}{m}$ denotes the matrix group of $m\times m$ orthogonal matrices, 
%(if $\mathbb F = \mathbb R$) or unitary matrices (if $\mathbb F = \mathbb C$) 
$\mu$ denotes the unique, left-invariant {\it Haar-measure} on $\U{F}{m}$, and $P$ is any $m \times m$ orthogonal projection onto a one-dimensional subspace of $\mathbb R^m$. 
In the following proposition, we provide an analytic expression for the $2$-coherence tensor.  In order to express its dependence on the underlying field and to simplify notation, we  define the constants
$$
\A{F}{m} := \frac{ \D{1}{F}{m} + (m-1)^2 }{ m^2 \D{1}{F}{m}}  \, \text{ and }
\B{F}{m} := \frac{\A{F}{m}}{3},
%\left\{
%	\begin{array}{cc}
%            \frac{\A{F}{m}}{3}, & \mathbb F = \mathbb R \\
%            \frac{\A{F}{m}}{2}, & \mathbb F = \mathbb C \\
%	\end{array}
%\right. 
%\text{ and  }
%\C{F}{m} :=
%\left\{
%	\begin{array}{cc}
%            \frac{\A{F}{m}}{3}, & \mathbb F = \mathbb R \\
%            0, & \mathbb F = \mathbb C \\
%	\end{array}
%\right., 
$$
and for each $j , j' \in \{1,2,...,n\}$, we denote the canonical matrix units for $\mathbb R^{m \times m}$ by $E_{j,j'} := e_j \otimes (e_{j'})^*.$

\begin{prop}\label{prop_K2comp}[see \cite[Ch. 7-9]{Ma2007} for details] The $2$-coherence tensor for $\mathbb R^m$ can be expressed in terms of the constants
$\A{F}{m}$ and $\B{F}{m}$ by
$$
\K{2}{F}{m} =
\A{F}{m} \sum\limits_{j=1}^{m} \ETE{j}{j}{j}{j}
+
 \B{F}{m} \sum\limits_{\tiny \substack{j,j' =1 \\ j \neq j'}}^m                                   
\Big(
         \ETE{j}{j'}{j'}{j} + \ETE{j}{j}{j'}{j'} 
        +
       \ETE{j}{j'}{j}{j'}
\Big).
$$
\end{prop}
\begin{proof} As defined, the $2$-coherence tensor is an average of rank one orthogonal projections, yielding the trace normalization $\TR(\K{2}{F}{m}) = 1$.  Its explicit form follows by the invariance properties of $\TR(\K{2}{F}{m})$
under the tensor representation of the orthogonal or unitary group (see \cite[Ch. 7-9]{Ma2007} for details).
\end{proof}

In light of the explicit expression for $\K{2}{F}{m}$ from Proposition~\ref{prop_K2comp}, the value of its squared Hilbert Schmidt norm follows by direct computation, which we record in the following corollary.
\begin{cor}\label{cor_K2norm}
The squared Hilbert-Schmidt norm of the $2$-coherence tensor is 
$$
  \HSNORM{\K{2}{F}{m}}^2 = \TR\left( (\K{2}{F}{m})^2 \right) = m \A{F}{m}^2 + m (m-1) \left(3 \B{F}{m}^2 \right) = \A{F}{m}.
$$
\end{cor}

With these basic properties of the $2$-coherence tensor established, next we compute  
$\mathbb E \left[ \left(\Q{1}{F}{m}(\omega) \right)^{\otimes 2} \right]$ and, in particular, provide the desired concrete expression of the second tensor embedding.

\begin{prop}\label{prop_Q2}
We have have $\VAR{2} = \K{2}{F}{m} - \SCITI$; in particular, the second embedding is given
by
$$
\Q{2}{F}{m}: \SPH{F}{m} \rightarrow \SA{F}{m}^{\otimes 2}: 
x \mapsto \left(\Q{1}{F}{m}(x)\right) \otimes \left(\Q{1}{F}{m}(x)\right)^*  - \K{2}{F}{m} + \SCITI.
$$
\end{prop}
\begin{proof}
Upon the expansion of $\Q{1}{F}{m}(\omega)^{\otimes 2}$,
$$
   \Q{1}{F}{m}(\omega) \otimes \Q{1}{F}{m}(\omega) =
   (\omega \otimes \omega^*)^{\otimes 2} 
   - {\frac{1}{m}} {\bf I}_m\otimes \omega \otimes \omega^* -{\frac{1}{m}} \omega \otimes \omega^* \otimes {\bf I}_m + \SCITI ,\, 
$$
computing the expectation term by term gives
$$
 \VAR{t} = \mathbb E[ \Q{1}{F}{m}(\omega) \otimes \Q{1}{F}{m}(\omega) ]
 = \K{2}{F}{m} -{\frac{2}{m^2}} {\bf I}_m \otimes {\bf I}_m +\SCITI \,  , 
$$   
where we have used $\mathbb E[ \omega\otimes \omega^*] =
{\scriptstyle\frac{1}{m}} {\bf I}_m$.
Simplifying yields the claimed identity.
\end{proof}

The orthogonality condition implied by Theorem~\ref{th_orth_cond} and Corollary~\ref{cor:QKorth} for the second tensor embedding thus reads as follows.
\begin{cor}\label{cor:QKorth}
For $t=2$, %$M=\K{2}{F}{m} -\SCITI$
$$
  \HSIP{ \Q{2}{F}{m}(x), \VAR{2}} =\HSIP{ \Q{2}{F}{m}(x), \K{2}{F}{m} - \SCITI  }=0 .
$$
\end{cor}

In the following lemma, we compute the Hilbert Schmidt inner product between an arbitrary $1$-embedded vector, $\Q{1}{F}{m}(x)$, and the $2$-coherence tensor, $\K{2}{F}{m}$ and show that $\Q{2}{F}{m}$ 
is indeed a norm-preserving embedding.  Afterward, we show how the inner products between the embedded vectors relate to cosine set of the original frame.

\begin{lem}\label{lem_K2rel}  
     Given a unit vector $x \in \mathbb F^m$, then
     \begin{enumerate}[(i)]
     \item\label{cl_1}  $\HSIP{ \left(\Q{1}{F}{m}(x)\right)^{\otimes 2}, \K{2}{F}{m}} = \A{F}{m} - \frac{1}{m^2}$
	%\item\label{cl_2}
	%	$\HSIP{\Q{2}{F}{m}(x), \K{t}{F}{m} - \SCITI} = 0$, and
	\item\label{cl_3}
		$\HSNORM{\Q{2}{F}{m}(x)}^2  = 1 - \frac{2}{m} + \frac{2}{m^2} - \A{F}{m}$, which shows that the second embedding is norm preserving,   up to a scale factor .
\end{enumerate}
\end{lem}
\begin{proof}
We deduce from $\Q{2}{F}{m}(x)=\left(\Q{1}{F}{m}(x)\right)^{\otimes 2} - \VAR{2}$
with $\TR\left[\Q{1}{F}{m}(x)\right]=\TR\left[\Q{1}{F}{m}(\omega)\right]=0$ 
that $\TR\left[\Q{2}{F}{m}(x)\right]=0$.
Together with the orthogonality condition,
$\HSIP{\Q{2}{F}{m}(x),\VAR{2}} = 0$,
this gives
$$
  \HSIP{ \Q{2}{F}{m}(x), \K{2}{F}{m} } 
  = \HSIP{ \Q{2}{F}{m}(x), \SCITI }
  = {\scriptstyle \frac{1}{m^2}} \TR\left[\Q{2}{F}{m}(x) \right] = 0 \, .
$$
Expressing $\Q{2}{F}{m}(x)$ in terms of the tensor power  $(\Q{1}{F}{m}(x))^{\otimes 2}$ 
then implies
$$
  \HSIP{(\Q{1}{F}{m}(x))^{\otimes 2}, \K{2}{F}{m}}
   - \HSIP{ \K{2}{F}{m} - \SCITI, \K{2}{F}{m}}
   = \TR\left[ (\K{2}{F}{m})^2\right] - {\scriptstyle \frac{1}{m^2}} \TR[ \K{2}{F}{m} ] 
   = \TR\left[ (\K{2}{F}{m})^2\right] - \frac{1}{m^2} \, .
$$

Note that the explicit form 
for $\K{2}{F}{m}$ in Proposition~\ref{prop_K2comp} implies the trace normalization
\begin{equation}\label{eq_trK2is1}
     \TR\left[ \K{2}{F}{m} \right] = m \A{F}{m} + m(m-1) \B{F}{m}=1
\end{equation}
and, from Corollary~\ref{cor_K2norm}, we have the identity for the (squared) Frobenius norm
\begin{equation}
    \TR\left[ (\K{2}{F}{m})^2\right] = \A{F}{m} \, .
\end{equation}
%implies the orthogonality relation,
%\begin{equation*}
%\HSIP{\SCITI, \K{2}{F}{m} - \SCITI} = \frac{1}{m^2} \TR\left( \K{2}{F}{m} \right) -  \frac{1}{m^4} \TR\left( {\bf I}_m \otimes {\bf I}_m \right) \\  
%                                                      =0.
%\end{equation*}
%In particular, the cross-terms of the squared norm in the left-hand side of~(\ref{cl_1}) vanish, yielding
%\begin{equation}\label{eq_lem1simp}
%\HSNORM{\K{2}{F}{m} -\SCITI}^2 = \HSNORM{\K{2}{F}{m}}^2 -\frac{1}{m^2},
%\end{equation}
%so the first claim follows by substituting Equation~\ref{?????} into Equation~\ref{eq_lem1simp}. 
%
%\jih{We need to record that $\HSNORM{\K{2}{F}{m}}^2 = \A{F}{m}$ somewhere and point the ??? above to it.}

%Upon the expansion of $\Q{1}{F}{m}$ in the the left-hand side of~(\ref{cl_1}), we obtain
%\begin{align*} 
% \HSIP{ \left(\Q{1}{F}{m}(x)\right)^{\otimes 2}, \K{2}{F}{m}}=
%\HSIP{\left(x \otimes x^* \right)^{\otimes 2},  \K{2}{F}{m}} 
% &- \frac{1}{m} \HSIP{\left(x\otimes x^*\right)\otimes {\bf I}_m, \K{2}{F}{m}}\\
% &- \frac{1}{m}  \HSIP{{\bf I}_m \otimes\left(x\otimes x^*\right), \K{2}{F}{m}}\\
%&+\frac{1}{m^2} \HSIP{{\bf I}_m \otimes {\bf I}_m, \K{2}{F}{m}}.\\
%\end{align*}
%By the invariance properties of $\K{2}{F}{m}$, there is no loss in generality in setting $x=e_1$, so it follows from Proposition~\ref{prop_K2comp} that
We conclude
\begin{equation}\label{eq_ipQt2K}
 \HSIP{ \left(\Q{1}{F}{m}(x)\right)^{\otimes 2}, \K{2}{F}{m}} %= \A{F}{m} - \frac{2}{m}(\A{F}{m} + (m-1)\B{F}{m}) + \frac{1}{m^2}= 
 = \A{F}{m} - \frac{1}{m^2}\, .
\end{equation}
%where the second equality follows by applying the normalization identity from Equation~(\ref{eq_trK2is1}).
This shows Claim~(\ref{cl_1}).

%\jih{Am I correctly stating and using the invariance properties of K2?}
%
%
%
%
%To see the second claim, note that
%the trace normalization identity from Equation~(\ref{eq_trK2is1})
%implies the orthogonality relation,
%\begin{equation*}
%\HSIP{\SCITI, \K{2}{F}{m} - \SCITI} = \frac{1}{m^2} \TR\left( \K{2}{F}{m} \right) -  \frac{1}{m^4} \TR\left( {\bf I}_m \otimes {\bf I}_m \right) \\  
%                                                      =0,
%\end{equation*}
%and that, by definition of $\Q{1}{F}{m}$, we have 
%\begin{equation*}
%\HSIP{\left(\Q{1}{F}{m}(x)\right)^{\otimes 2}, \SCITI} = 0.
%\end{equation*}
%Thus, after expanding $\Q{2}{F}{m}$ in the left-hand side of Claim~(\ref{cl_2}) and accounting for the vanishing terms just mentioned, we obtain
%\begin{equation}\label{eq_cl2simp}
%     \HSIP{\Q{2}{F}{m}(x), \K{t}{F}{m} - \SCITI} 
%      = \HSIP{ \left(\Q{1}{F}{m}(x)\right)^{\otimes 2}, \K{2}{F}{m}}  
%       - \HSNORM{\K{2}{F}{m}}^2 + \HSIP{\K{2}{F}{m}, \SCITI}, 
%\end{equation}
%so Claim~(\ref{cl_2}) follows by substituting Equation~(\ref{eq_ipQt2K}), Corollary~(\ref{???}) and Equation~(\ref{eq_trK2is1}) into Equation~(\ref{eq_cl2simp}).

%\jih{A corollar to the K2 prop should stated that the squared norm of K2 equals $\A{F}{m}$.  The ??? above should point to it.}

To see the second claim, we re-express the squared Hilbert Schmidt norm of $\Q{2}{F}{m}$ as an inner product, substitute $\Q{2}{F}{m}$ in the left side of the inner product with its definition, and apply the orthogonality relation, which gives
\begin{equation}\label{eq_IPQ1t2Q2}
     \HSIP{\Q{2}{F}{m}(x), \Q{2}{F}{m} (x)} 
=
\HSIP{\left(\Q{1}{F}{m}(x)\right)^{\otimes 2}  - \VAR{2}, \Q{2}{F}{m}(x)} 
=
\HSIP{\left(\Q{1}{F}{m}(x)\right)^{\otimes 2}, \Q{2}{F}{m}(x)} .
\end{equation}
Next, we replace $\Q{2}{F}{m}$ with its definition in the right side of the inner product to obtain
\begin{equation}\label{eq_ipQ1t2Q2}
\HSIP{\left(\Q{1}{F}{m}(x)\right)^{\otimes 2}, \Q{2}{F}{m}} 
=
\HSNORM{\Q{1}{F}{m}(x)}^4 - \HSIP{\left(\Q{1}{F}{m}(x)\right)^{\otimes 2}, \K{2}{F}{m} - \SCITI }.
\end{equation}
The first term on the right-hand side simplifies to
$$
  \HSNORM{\Q{1}{F}{m}(x)}^4 
=
\left(\TR\left[ \left(x \otimes x^* - {\scriptstyle \frac 1 m} {\bf I}_m\right)^2 \right] \right)^2
  = 
\left(\TR\Big[ \left(x \otimes x^* - {\scriptstyle \frac 1 m} {\bf I}_m\right) x \otimes x^* \Big] \right)^2 
=
\left(\frac {m-1}{m}\right)^2.
$$
Using that  $\TR\left( \Q{1}{F}{m}\right)=0$ by definition, the second  term in the right-hand side of Equation~(\ref{eq_ipQ1t2Q2}) resolves to
\begin{equation}\label{eq_ipQ1t2K2minussciti}
\HSIP{\left(\Q{1}{F}{m}(x)\right)^{\otimes 2}, \K{2}{F}{m} - \SCITI } 
=
 \HSIP{\left(\Q{1}{F}{m}(x)\right)^{\otimes 2}, \K{2}{F}{m}} 
= \A{F}{m} - \frac{1}{m^2},
\end{equation}
where the second equality follows by the first claim of this lemma.  Thus, Claim~(\ref{cl_3}) follows by combining the two terms.

\end{proof}

Finally, we present
the desired equation, which governs the relationship between a frame's cosine set and the corresponding set of signed angles between the higher-dimensional embedded vectors.

\begin{thm}\label{thm_tier2}
	Given unit vectors $\{f_j\}_{j=1}^n \subset \mathbb R^m$, $m \ge 2$, there exist corresponding 
	unit vectors $\left\{T^{(2)}_j \right\}_{j=1}^n \subset \mathbb R^{\D{2}{F}{m}}$ such that for any $j, l \in \{1,2,...,n\}$, the following equation holds:
	$$
\left\langle T^{(2)}_j, T^{(2)}_l \right\rangle  =
\frac{m^2 \D{1}{F}{m}}{(\D{1}{F}{m}-1)(m-1)^2}
\Bigg( \left(|\langle f_j, f_l \rangle|^2 - \frac 1 m\right)^2 - \left( \A{F}{m} - \frac{1}{m^2}\right)  \Bigg) \, .
	$$
\end{thm}
\begin{proof}
Letting $\tilde T^{(2)}_j=\frac{\Q{2}{F}{m}(f_j)}{\HSNORM{\Q{2}{F}{m}(f_j)}}$ 
and 
$\tilde T^{(2)}_l=\frac{\Q{2}{F}{m}(f_l)}{\HSNORM{\Q{2}{F}{m}(f_l)}}$, we may isomorphically identify them with corresponding unit vectors $T^{(2)}_j, T^{(2)}_l \in \SPH{R}{\D{2}{F}{m}}$.
Thus, in terms of these vectors, we have
\begin{equation}\label{eq_iso1}
\NORM{T^{(2)}_j - T^{(2)}_l}^2 =2 - 2 \left\langle T^{(2)}_j, T^{(2)}_l \right\rangle.
\end{equation}
On the other hand, we may pass back to their tensored forms to obtain
\begin{equation}\label{eq_iso2}
\NORM{T^{(2)}_j - T^{(2)}_l}^2
 = \HSNORM{\tilde T_j -\tilde T_l}^2 
= 2 - \frac{2}{\HSNORM{\Q{2}{F}{m}(f_j)}\HSNORM{\Q{2}{F}{m}(f_l)}} \HSIP{\Q{2}{F}{m}(f_j), \Q{2}{F}{m}(f_l)},
\end{equation}
so by Claim~(\ref{cl_3}) of Lemma~\ref{lem_K2rel}  , Equation~(\ref{eq_iso2}) reduces to
\begin{equation}\label{eq_iso2a}
\NORM{T^{(2)}_j - T^{(2)}_l}^2
= 2 - \frac{2}{1 - \frac 2 m + \frac{2}{m^2} - \A{F}{m}} \HSIP{\Q{2}{F}{m}(f_j), \Q{2}{F}{m}(f_l)}.
\end{equation}
By inserting the expression for $\A{F}{m}$, we simplify further to
\begin{equation}\label{eq_iso3}
\NORM{T^{(2)}_j - T^{(2)}_l}^2
= 2 - \frac{2m^2\D{1}{F}{m}}{(m-1)^2(\D{1}{F}{m}-1)} \HSIP{\Q{2}{F}{m}(f_j), \Q{2}{F}{m}(f_l)}.
\end{equation}

Next, we expand $\Q{2}{F}{m}(f_l)$ and obtain
\begin{equation}\label{eq_iso4}
\HSIP{\Q{2}{F}{m}(f_j), \Q{2}{F}{m}(f_l)} 
= \HSIP{\Q{2}{F}{m}(f_j), \left(\Q{1}{F}{m}(f_l)\right)^{\otimes 2}} - 
\HSIP{\left(\Q{1}{F}{m}(f_l)\right)^{\otimes 2}, \K{2}{F}{m} - \SCITI}, 
\end{equation}
By  Corollary~\ref{cor:QKorth}, the second additive term on the right-hand side of Equation~\ref{eq_iso4} vanishes, so expanding $\Q{2}{F}{m}(f_j)$ gives
\begin{equation}\label{eq_iso5}
\HSIP{\Q{2}{F}{m}(f_j), \Q{2}{F}{m}(f_l)} = 
\HSIP{\left(\Q{1}{F}{m}(f_j)\right), \left(\Q{1}{F}{m}(f_l)\right)}^2 
-
\HSIP{\K{2}{F}{m} - \SCITI, \left(\Q{1}{F}{m}(f_l)\right)^{\otimes 2}}. 
\end{equation}
Because $\TR\left(\Q{1}{F}{m}\right)=0$, the second additive term on the right-hand side of Equation~(\ref{eq_iso5}) simplifies to
\begin{equation}\label{eq_iso6}
\HSIP{\K{2}{F}{m} - \SCITI, \left(\Q{1}{F}{m}(f_l)\right)^{\otimes 2}}
=
\HSIP{\K{2}{F}{m}, \left(\Q{1}{F}{m}(f_l)\right)^{\otimes 2}}
=
\A{F}{m} - \frac{1}{m^2},
\end{equation}
where the second equality follows from the first claim of Lemma~\ref{lem_K2rel} .
Expanding $\Q{1}{F}{m}(f_j)$ and $\Q{1}{F}{m}(f_l)$ within the square of the first additive term on the right-hand side of Equation~(\ref{eq_iso5}) yields
\begin{equation}\label{eq_iso7}
\HSIP{\left(\Q{1}{F}{m}(f_j)\right), \left(\Q{1}{F}{m}(f_l)\right)}^2  =\left( |\langle f_j, f_l \rangle|^2 - \frac 1 m \right)^2,
\end{equation}
so substituting Equation~(\ref{eq_iso7}) and Equation~(\ref{eq_iso6}) into Equation~(\ref{eq_iso5}) gives
\begin{equation}\label{eq_iso8}
\HSIP{\Q{2}{F}{m}(f_j), \Q{2}{F}{m}(f_l)} = \left( |\langle f_j, f_l \rangle|^2 - \frac 1 m \right)^2 - \left(\A{F}{m} - \frac{1}{m^2} \right).
\end{equation}
The claim follows by equating Equation~(\ref{eq_iso1}) with Equation~(\ref{eq_iso3}), replacing the inner product with the value given in Equation~(\ref{eq_iso8}), and then rearranging to the desired form.
\end{proof}

%and apply Claim~(\ref{cl_3}) of Lemma~\ref{lem_K2rel} to obtain the normalized vectors
%$$
%T_j={\scriptstyle \frac{1}{\sqrt{1 - \frac 2 m + \frac{2}{m^2} -\A{F}{m}}}}\Q{2}{F}{m}(f_j)
% \text{ and } 
%T_l={\scriptstyle \frac{1}{\sqrt{1 - \frac 2 m + \frac{2}{m^2} -\A{F}{m}}}}\Q{2}{F}{m}(f_l).
%$$

%Unlike the case $t=1$, it is evident from Theorem~\ref{thm_tier2} that the embedding equation for the case $t=2$ depends on the underlying field.  Noting that
%$$
%\A{R}{m} = \frac{3}{m(m+2)}, \, \D{1}{R}{m} = 
%		\frac{(m+2)(m-1)}{2} , \,
% \A{C}{m} = \frac{2}{m(m+1)}, \, \text{ and }
%	\D{1}{C}{m} 	=
%		m^2-1 ,
%		$$
%we re-express the two distinct cases of Theorem~\ref{thm_tier2} as  corollaries.

We re-express this governing equation from Theorem~\ref{thm_tier2} in terms of the ambient dimension, $m$.
\begin{cor}\label{cor_realtier2}
Given unit vectors $\{f_j\}_{j=1}^n \subset \mathbb R^m$, there exist corresponding 
	unit vectors $\left\{T_j^{(2)}\right\}_{j=1}^n \subset \mathbb R^{\D{2}{R}{m}}$ such that for any $j, l \in \{1,2,...,n\}$, the following equation holds:
	$$
\left\langle T^{(2)}_j, T^{(2)}_l \right\rangle  =
\frac{m^2(m+2)}{m^3 -5m +4}
\Bigg( \left(|\langle f_j, f_l \rangle|^2 - \frac 1 m\right)^2 -\frac{2(m-1)}{m^2(m+2)}  \Bigg) \, .
	$$
\end{cor}

%\input{cor_complextier2}

%Because $\Q{2}{F}{m}$ is an isometry, if it maps a unit norm frame, $\mathcal F=\{f_j\}_{j=1}^n$ to an optimal configuration of caps, (ie it saturates one of the two bounds in Theorem~\ref{thm_rankin}), then it follows that $\mathcal F$ is a Grassmannian frame.

We conclude by combining this embedding with Rankin's bound in order to characterize Grassmannian frames.

\begin{thm}\label{thm:main}
Given a frame $\mathcal F$ of $n$ vectors in ${\mathbb R}^m$, with $m \ge 2$, $\frac{n-\D{1}{F}{m}}{n-1} \ge \frac{\D{1}{F}{m}}{(m-1)^2}$, and $\left\{\Q{2}{F}{m}(f): f \in {\mathcal F}\right\}$
forms a simplex in a subspace of $\SA{F}{m}^{\otimes 2}$, then $\mathcal F$ is Grassmannian. Moreover, if $n> \D{2}{F}{m}+1$ and $(m-1)^2 \ge \D{1}{F}{m}$,
and the inner products between pairs of $\left\{\Q{2}{F}{m}(f): f \in {\mathcal F} \right\}$ are non-positive, then the frame is Grassmannian.
\end{thm}
\begin{proof}
%We first consider the complex case.
%Unfortunately, the assumption $  \frac{n-\D{1}{C}{m}}{n-1} \ge \frac{\D{1}{C}{m}}{(m-1)^2}$ implies for $m\ge 2$
%that $\frac{n-\D{1}{C}{m}}{n-1} \ge 1$ which rules out $\D{1}{C}{m}>1$, so this case is vacuous.
%Next, we treat the real case.
If $\frac{n-\D{1}{R}{m}}{n-1} \ge \frac{\D{1}{F}{m}}{(m-1)^2} $, then by estimating $\D{1}{F}{m} > (m-1)^2/2$,
we have
$n > 2\D{1}{R}{m} - 1$ and with $\D{1}{R}{m} \ge 2$, we get
$n >\D{1}{R}{m} + 1$, so the orthoplex bound holds, $\max_{j \ne l} |\langle f_j, f_l \rangle |^2 \ge \frac 1 m$.
Because the polynomial $p(x) = \frac{m^2\D{1}{F}{m}}{(\D{1}{F}{m}-1)(m-1)^2}\left((x- \frac 1 m)^2 - (\A{F}{m}- \frac{1}{m^2})\right)$
is decreasing on $[0, \frac 1 m]$ and increasing on $\left[\frac 1 m, \infty\right)$,
%we know that minimizing 
%$p\left( \max_{j \ne l} |\langle f_j, f_l \rangle |^2\right)$ over $\UNFs{R}$
%is equivalent to minimizing $\max_{j \ne l} |\langle f_j, f_l \rangle |$ over $\UNFs{R}$.
%With the monotonicity of $p$ on $[0,\frac 1 m]$, we have
it follows that
$$
   \max_{j \ne l} p\left( |\langle f_j, f_l \rangle |^2\right) \le \max \left\{p(0), p(\max_{j \ne l} |\langle f_j, f_l \rangle |^2) \right\} \, .
$$

If $\frac{n-\D{1}{F}{m}}{n-1} \ge \frac{\D{1}{F}{m}}{(m-1)^2} $, then 
$ p(0) \le - \frac{1}{n-1}$, and Rankin's bound implies that
$$p\left(\max_{j \ne l} |\langle f_j, f_l \rangle |^2\right) \ge - \frac{1}{n-1}.$$
Hence, if $\left\{\Q{2}{F}{m}(f): f \in {\mathcal F}\right\}$ forms a simplex then equality is achieved and the frame is Grassmannian.

We continue with the more restrictive assumption
$n>\D{2}{F}{m}+1$,
where we know Rankin's strengthened bound holds. In this case, assuming
$(m-1)^2 \ge \D{1}{F}{m}$ implies
$p(0) \le 0$, which, by Rankin's bound, implies $p\left(\max_{j \ne l} |\langle f_j, f_l \rangle |^2\right) \ge 0$ and if all the Hilbert-Schmidt inner
products between pairs of $\{\Q{2}{F}{m}(f): f \in {\mathcal F}\}$ are non-positive, then equality holds and the frame is Grassmannian.
%
%By Rankin's bound, the minimum of $p\left(\max_{j \ne l} |\langle f_j, f_l \rangle |^2\right)$ cannot be smaller than $-\frac{1}{n-1}$. Consequently,
%if $p( \max_{j \ne l} |\langle f_j, f_l \rangle |^2) = - \frac{1}{n-1}$, then the frame $\mathcal F$ is Grassmannian.
\end{proof}

\section{Examples of high redundancy frames with low coherence arising from the second embedding}
To conclude this work, we present examples of frames with low coherence that arise from second tensor embedding and discuss their interesting structural properties.

\begin{ex}
A set of orthonormal bases, $\left\{\mathcal B_j\right\}_{j=1}^k$ for $\mathbb R^m$ are said to be {\bf mutually unbiased} if $|\langle x, y \rangle|^2 = \frac 1 m$ for every $x \in \mathcal B_j, y \in \mathcal B_l$ with $j \neq l$.
The existence of three such bases in $\mathbb R^4$ is well-known~\cite{ConwayHardinSloane1996, Appleby2009, MR3557826} and it is also known that their union is a Grassmannian frame~\cite{ConwayHardinSloane1996, Appleby2009, MR3557826}, in accordance with the conditions of Theorem~\ref{thm_t1}.  Out of curiousity, we fed this system of vectors through equation from Corollary~\ref{cor_realtier2} and discovered that the embedded bases, $\left\{ \mathcal B_j^{(2)} \right\}_{j=1}^3$, have the following peculiar property.  

     Given any choice of $j \in \{1,2,3\}$ and any vector $T^{(2)}_j \in \mathcal B_j^{(2)}$, we observe that orthogonal vectors remain orthogonal when embedded and
$$
\Big\{ 
\left\langle  T^{(2)}_j , T^{(2)}_l \right\rangle  : T^{(2)}_l \in \mathcal B^{(2)}_l, l \in \{1,2,3\}, l\neq j
\Big\} 
=
\Bigg\{ - \frac 1 8 \Bigg\}, 
$$
so that each embedded vector resembles the vertex of a 9-simplex relative to the eight vectors coming from the other two embedded bases.
\end{ex}

\begin{ex}[16 vectors in $\mathbb R^3$]
Assisted by Sloane's database of putatively optimal packings~\cite{sloanetbl}, we confirmed numerically that the second embedding maps Sloane's example of 16 vectors in $\mathbb R^3$ into a packing of 16 $\frac{\pi}{2}$-caps.  After a helpful discussion with Dustin Mixon, we then ascertained that this example corresponds to the $16$ lines passing through the antipodal vertices of a {\bf biscribed pentakis dodecahedron}.  Analytic coordinate representations of the $32$ vertices of this polytope can be found at {\it Visual Polyhedra}~\cite{vispoly}, an online database of various exotic polyhedra in $\mathbb R^3$.  After discarding antipodal points, the remaining $16$ vertices correpond to a $6$-angular unit norm tight frame, $\mathcal F$, for $\mathbb R^3$ with cosine set
$$
\Theta_{\mathcal F}= 
\left\{
\scriptstyle 
\sqrt{ \frac{1}{15} (5 - 2 \sqrt 5)},
 \frac{1}{\sqrt 5},
\frac 1 3,
\sqrt{\frac{7}{15}},
 \frac{\sqrt 5}{3},
\sqrt{ \frac{1}{15} (5 + 2 \sqrt 5)}
\right\}. % d tensor embedding.
$$
Passing these values through the equation from Corollary~\ref{cor_realtier2} shows that the 16 frame vectors of $\mathcal F$ map to 16 vectors, $\mathcal T = \{T_j\}_{j=1}^{16}$, on the sphere in $\mathbb R^{14}$ for whose (signed) cosine set is
   $$
\left\{ \langle T^{(2)}_j, T^{(2)}_l \rangle : j\neq l \right\} = \{-1/5, -1/9, 0\},
$$
meaning the embedded vectors form the centers of an optimal cap-packing according to the second of Rankin's conditions in Theorem~\ref{thm_rankin}. 
Unfortunately, the parameters of this example do not satisfy the sufficiency conditions from Theorem~\ref{thm:main}, so we may not state with certainty this is indeed a Grassmannian frame; however, curiously, we have also observed that the first embedding maps this frame into a set of vectors, $\left\{T^{(1)}_j\right\}_{j=1}^{16} \subset \mathbb R^5$, which forms a tight Grassmannian frame, characterized by the original orthoplex bound.  The optimal incoherence of this intermediate frame has recently been observed by Fickus, Jasper, and Mixon~\cite{2017arXiv170701858F}.  Given the evidence, we find it reasonable to posit that  $\mathcal F$ is likely a Grassmannian frame.  If so, then it would follow that 
$$\mu_{16, 3}(\mathbb R)= \sqrt{ \frac{1}{15} (5 + 2 \sqrt 5)}.$$
\end{ex}

\begin{ex}[120 vectors in $\mathbb R^8$]
The optimal coherence of this example has been verified via the Levenschtein bound~\cite{lev_paper}, but we recertify it here in terms of the second tensor embedding. After discarding antipodal vectors from the 240 shortest vectors of the E8 lattice, the remaining $120$ vectors form a unit norm frame, $\mathcal F$, for $\mathbb R^8$ with cosine set,
$$
\Theta_{\mathcal F} = \{0, 1/2 \}.
$$
We verify the assumptions of Theorem~\ref{thm:main}: $\D{1}{R}{8}=35$, so $\frac{120-\D{1}{R}{8}}{119} = \frac{85}{119} \ge \frac{35}{49} = \frac{\D{1}{R}{8}}{(8-1)^2}$.
One can verify that the second tensor embedding maps this frame to a regular $119$-simplex, meaning the embedded vectors correspond to an optimal cap packing according to the first of Rankin's conditions from Theorem~\ref{thm_rankin}, thereby verifying that $\mathcal F$ is a Grassmannian frame and $$\mu_{120,8} (\mathbb R)=\frac{1}{2}.$$ 
\end{ex}

\bibliography{tier2_bib}
\bibliographystyle{plain}

\end{document}